\documentclass[11pt,a4paper]{article}
\usepackage[T1]{fontenc}
\usepackage{lmodern}
\usepackage{amsmath,amssymb,amsthm}
\usepackage[margin=1in]{geometry}
\usepackage{microtype}
\usepackage{enumitem}
\usepackage[hidelinks]{hyperref}
\usepackage[nameinlink,noabbrev]{cleveref}
\newcommand{\R}{\mathbb R}
\newcommand{\M}[2]{M_{#1,#2}(\R)}
\DeclareMathOperator{\interior}{int}
\DeclareMathOperator{\rank}{rank}
\DeclareMathOperator{\GL}{GL}
\DeclareMathOperator{\tr}{tr}

\newtheorem{theorem}{Theorem}[section]
\newtheorem{lemma}[theorem]{Lemma}
\newtheorem{proposition}[theorem]{Proposition}
\newtheorem{corollary}[theorem]{Corollary}
\theoremstyle{definition}
\newtheorem{definition}[theorem]{Definition}
\newtheorem{example}[theorem]{Example}
\theoremstyle{remark}
\newtheorem{remark}[theorem]{Remark}
\numberwithin{equation}{section}
\setlist[enumerate]{itemsep=4pt,topsep=5pt}
\setlist[itemize]{itemsep=3pt,topsep=5pt}
\makeatletter
\def\@cite#1#2{[{#1\if@tempswa ,~#2\fi}]}
\makeatother
\title{Linear Preservers of Semipositivity and Minimal Semipositivity}
\author{Yulong Tian\thanks{Corresponding author.}\quad and\quad Jinli Xu}
\date{}
\hypersetup{pdftitle={Linear Preservers of Semipositivity and Minimal Semipositivity},
  pdfauthor={Yulong Tian and Jinli Xu},
  pdfsubject={Semipositivity over proper cones and square minimal semipositivity over the nonnegative orthant}}

\begin{document}
\maketitle
\begin{abstract}
We classify all real linear into preservers of semipositivity on
$\M{m}{n}$, $m,n\ge2$, relative to arbitrary source and target pairs of
proper cones, without assuming invertibility or a prescribed form.
The classification consists of product maps $L(A)=XAY$ and rank-$n$
degenerate maps $L(A)=\Psi(u^TA)$, subject to explicit cone and
matrix-pencil conditions.
For square minimal semipositivity with respect to the standard
nonnegative orthant, preservation forces the operator to be
invertible: for $n\ge2$, the linear into preservers are precisely the
maps $L(A)=XAY$ and $L(A)=XA^TY$ with $X,Y\in\GL_n(\R)$ satisfying
$X^{-1},Y^{-1}\ge0$.
We also classify all linear preservers of orthant semipositivity
from $\M{2}{2}$ to $\M{1}{3}$ into three disjoint classes,
two of which consist of maps of rank three outside product form.
\end{abstract}



\section{Introduction and main results}

The linear preserver problem in matrix theory asks for the
characterization of linear maps between matrix
spaces that leave a given subset or property invariant; see the survey
by Li and Pierce~\cite{lipierce2001}.
Given a subset $\mathcal S$ of a matrix space $V$, one may ask for all
linear maps $L:V\to V$ with $L(\mathcal S)\subseteq\mathcal S$,
or for all those with $L(\mathcal S)=\mathcal S$.
The former maps are called \emph{into preservers}, while the latter
are called \emph{onto} or \emph{strong preservers}.
For many positivity classes of matrices (surveyed recently by Choudhury
and Yadav~\cite{choudhuryyadav2026}), strong preservers are by now well
understood. By contrast, characterizing into preservers is more
delicate because singular maps can arise; consequently, existing
classifications frequently rely on additional assumptions such as
invertibility or a prescribed form.

Recall that a real $m\times n$ matrix $A$ is \emph{semipositive}
if there exists a vector $x>0$ such that $Ax>0$, where the inequalities
$\ge0$ and $>0$ for vectors or matrices denote entrywise nonnegativity
and strict positivity, respectively. A semipositive matrix is
\emph{minimally semipositive} if no column can be deleted without losing
semipositivity. Semipositive matrices arise naturally in the theory of
nonnegative matrices and M-matrices~\cite{bermanplemmons1994}, and their
structural and geometric properties have been studied in
\cite{johnsonkerrstanford1994,tsatsomeros2016}.

For semipositivity on $\M{m}{n}$, into preservers of the product form
$A\mapsto XAY$ were characterized in
\cite{choudhurykannansivakumar2018,dorsey2016}.
Dorsey et al.~\cite{dorsey2016} further classified all onto preservers
and conjectured that every invertible into preserver has this product
form. Jayaraman and Mer proved this conjecture for $m\ge n\ge2$
\cite[Theorem~3.11]{jayaramanmer2021}, and Mer subsequently established
it for all $m,n\ge2$ \cite[Theorem~1.2]{mer2026}.
However, as observed in~\cite{dorsey2016}, noninvertible into preservers
need not have product form, leaving open the problem of classifying
all linear into preservers without an invertibility assumption.

More generally, semipositivity can be formulated relative to proper
cones by requiring a matrix to map some interior point of an input
cone into the interior of an output cone. Arumugasamy, Jayaraman, and
Mer~\cite{arumugasamy2018} developed this framework for linear maps
between finite-dimensional real Hilbert spaces. For $m\ge n\ge2$,
Jayaraman and Mer subsequently characterized product maps preserving
semipositivity with respect to a single fixed pair of proper cones
\cite[Theorem~3.29]{jayaramanmer2021}.

In this paper, we classify linear into preservers in two settings,
without assuming invertibility or a prescribed form.
First, for all $m,n\ge2$, we classify the linear maps on $\M{m}{n}$
that carry semipositive matrices relative to one pair of proper cones
into those relative to another pair. Second, for each $n\ge2$, we
classify all linear into preservers of minimally semipositive
$n\times n$ matrices with respect to the standard nonnegative orthant.

For positive integers $m,n$, let $\M{m}{n}$ denote the space of real
$m\times n$ matrices and $\GL_n(\R)$ the group of invertible
$n\times n$ matrices. Recall that a \emph{proper cone} in $\R^q$
is a closed convex cone $Q$ that is pointed ($Q\cap(-Q)=\{0\}$)
and has nonempty interior $\interior Q$; its dual cone is denoted by $Q^*$.

\begin{definition}\label{def:semipositive}
For proper cones $P\subseteq\R^n$ and $Q\subseteq\R^m$, the set of
\emph{$(P,Q)$-semipositive} matrices is defined by
\[
 S(P,Q)=\{A\in\M{m}{n}:Ax\in\interior Q
                 \text{ for some }x\in\interior P\}.
\]
\end{definition}

Our first main result shows that, for $m,n\ge2$, every linear into
preserver between proper-cone semipositivity classes in $\M{m}{n}$
is a product map $L(A)=XAY$ or a rank-$n$ degenerate map
$L(A)=\Psi(u^TA)$.
We express the cone condition for the degenerate form in terms of the
following \emph{matrix pencil}.
For a linear map $\Psi:\M{1}{n}\to\M{m}{n}$ and each $d\in\R^m$,
the map $r\mapsto d^T\Psi(r)$ is a linear endomorphism of $\M{1}{n}$.
Hence it is right multiplication by a unique matrix
$\mathcal Y_\Psi(d)\in\M{n}{n}$:
\begin{equation}\label{eq:pencil-definition}
 d^T\Psi(r)=r\mathcal Y_\Psi(d)
 \qquad(d\in\R^m,\ r\in\M{1}{n}).
\end{equation}
The matrix $\mathcal Y_\Psi(d)$ depends linearly on $d$; we write
$\mathcal Y(d)$ when $\Psi$ is fixed.

\begin{theorem}\label{thm:classification}
Let $m,n\ge2$ and let $K,H\subseteq\R^n$ and
$C,D\subseteq\R^m$ be proper cones.
A real linear map $L:\M{m}{n}\to\M{m}{n}$ satisfies
\begin{equation}\label{eq:preservation}
 L\bigl(S(K,C)\bigr)\subseteq S(H,D)
\end{equation}
if and only if it belongs to at least one of the following two classes.
\begin{enumerate}[label=\textup{(\Roman*)},ref=\Roman*]
\item There are $X\in\M{m}{m}$ and
$Y\in\GL_n(\R)$ such that
\begin{equation}\label{eq:product-class}
 L(A)=XAY\quad\text{for every }A,\qquad
 X(\interior C)\subseteq\interior D,\qquad Y^{-1}K\subseteq H.
\end{equation}
\item There are $0\ne u\in C^*$ and
a linear map $\Psi:\M{1}{n}\to\M{m}{n}$ such that
\begin{equation}\label{eq:degenerate-class}
 L(A)=\Psi(u^TA)\quad\text{for every }A,
\end{equation}
and the pencil $\mathcal Y(d)=\mathcal Y_\Psi(d)$ from
\eqref{eq:pencil-definition} satisfies
\begin{equation}\label{eq:full-dual-condition}
 \mathcal Y(d)\in\GL_n(\R),\qquad \mathcal Y(d)^{-1}K\subseteq H
 \qquad\text{for every }0\ne d\in D^*.
\end{equation}
\end{enumerate}
The two classes have the following rank and intersection properties.
\begin{enumerate}[label=\textup{(\roman*)}]
\item In \textup{(I)}, $\rank L=n\rank X$; in \textup{(II)}, $\Psi$ is
injective and $\rank L=n$.
\item The intersection of the two classes consists of those maps in
\textup{(I)} with $\rank X=1$.
\end{enumerate}
\end{theorem}

We next consider minimal semipositivity in $\M{n}{n}$, $n\ge2$,
with respect to the standard nonnegative orthant.
A square matrix is minimally semipositive if and only if it is
invertible with a nonnegative inverse
\cite{choudhurykannansivakumar2018}.
Thus the class under consideration is
\[
 \mathcal I_n=\{A\in\GL_n(\R):A^{-1}\ge0\}.
\]
For maps of the prescribed forms $A\mapsto XAY$ and
$A\mapsto XA^TY$, the necessary and sufficient conditions on $X$
and $Y$ are known
\cite[Theorems~2.10 and~2.11]{choudhurykannansivakumar2018}.
Our second main result shows that every linear into preserver of
$\mathcal I_n$ is automatically invertible and has one of these two
standard forms.

\begin{theorem}\label{thm:msp-classification}
Let $n\ge2$ and let $L:\M{n}{n}\to\M{n}{n}$ be a real linear map.
Then $L(\mathcal I_n)\subseteq\mathcal I_n$ if and only if there exist
$X,Y\in\GL_n(\R)$, with $X^{-1},Y^{-1}\ge0$, such that either
\[
 L(A)=XAY\quad\text{for every }A\in\M{n}{n},
\]
or
\[
 L(A)=XA^TY\quad\text{for every }A\in\M{n}{n}.
\]
In particular, every such into preserver is invertible.
\end{theorem}

The conclusion of Theorem~\ref{thm:classification} that every
preserver of rank greater than $n$ has product form does
not extend in general to maps between matrix spaces of different sizes.
For orthant semipositivity, Theorem~\ref{thm:additional-column}
classifies all linear preservers from $\M{2}{2}$ to $\M{1}{3}$ into
three disjoint classes. Two of these classes consist of maps of
rank three that lie outside product form.

The paper is organized as follows.
Sections~\ref{sec:cone-geometry}--\ref{sec:classification-proof}
establish the proper-cone classification in Theorem~\ref{thm:classification}
and its invertible case. Section~\ref{sec:classification-proof} also gives
an example showing that the pencil condition cannot in general be
checked only on extreme rays.
Section~\ref{sec:examples} gives the orthant specialization and the
classification from $\M{2}{2}$ to $\M{1}{3}$.
Section~\ref{sec:minimal-semipositivity} proves the classification of
linear into preservers of square minimally semipositive matrices.

\section{Cone geometry and a strict alternative}
\label{sec:cone-geometry}

Throughout Sections~\ref{sec:cone-geometry}--\ref{sec:classification-proof},
the dimensions and cones are as in Theorem~\ref{thm:classification}.
In these sections, a linear map $L:\M{m}{n}\to\M{m}{n}$ is said to
\emph{preserve semipositivity} if it satisfies \eqref{eq:preservation}.
We collect the separation and duality facts needed below and provide
short, self-contained proofs for some of them.
We use the Euclidean inner product $\langle x,y\rangle=x^Ty$.
For a subspace $V\subseteq\R^q$, write
\[
 V^\perp=\{x\in\R^q:x^Tv=0\text{ for every }v\in V\}.
\]
For a cone $Q\subseteq\R^q$, its dual cone is
\[
 Q^*=\{v\in\R^q:v^Tz\ge0\text{ for every }z\in Q\}.
\]

\begin{lemma}
\label{lem:cone-facts}
Let $Q\subseteq\R^q$ be a proper cone, with $q\ge1$. Then
\[
 0\notin\interior Q,\qquad
 \interior Q\cap(-\interior Q)=\varnothing,\qquad
 Q+\interior Q\subseteq\interior Q,\qquad
 \overline{\interior Q}=Q.
\]
The interior of $Q$ is convex and connected. Moreover, $Q^*$ is a proper
cone, and
\begin{equation}
 v^Ta>0\qquad
 (0\neq v\in Q^*,\ a\in\interior Q).
 \label{eq:strict-dual-positivity}
\end{equation}
Both $\interior Q$ and $\interior Q^*$ contain a basis of $\R^q$.
\end{lemma}

\begin{proof}
The first two assertions follow from $Q\cap(-Q)=\{0\}$.
Since $Q$ has nonempty interior, the remaining interior properties follow
from the cone property and \cite[Theorems~6.1--6.3]{rockafellar1970}.
The bipolar and dual-interior results in
\cite[Chapter~1, Exercises~2.3 and~2.8]{bermanplemmons1994}
show that $Q^*$ is proper. Applying
\cite[Chapter~1, (2.9)]{bermanplemmons1994} to $Q^*$,
with $Q^{**}=Q$, gives
\eqref{eq:strict-dual-positivity}.
Finally, every nonempty open subset of $\R^q$ contains a basis, since a
proper linear subspace contains no open ball.
\end{proof}

\begin{lemma}[{cf.~\cite[Theorem~11.2]{rockafellar1970}}]
\label{lem:open-cone}
Let $U$ be a nonempty open convex subset of a finite-dimensional Euclidean
space $E$, with $0\notin U$. Then there is $0\neq v\in E$ such that
$\langle v,u\rangle\geq0$ for every $u\in U$.
\end{lemma}

\begin{lemma}[{cf.~\cite[Theorem~11.2]{rockafellar1970}}]
\label{lem:annihilator}
Let $Q\subseteq\R^q$ be a proper cone and $V\subseteq\R^q$ a linear
subspace. If $V\cap\interior Q=\varnothing$, then
\[
 Q^*\cap V^\perp\neq\{0\}.
\]
\end{lemma}

\begin{proof}
The set $U=\interior Q-V$ is nonempty, open, and convex, and the
hypothesis gives $0\notin U$. By Lemma~\ref{lem:open-cone}, there is
$0\ne v\in\R^q$ such that
\[
 v^T(a-x)\ge0
 \qquad(a\in\interior Q,\ x\in V).
\]
For fixed $a\in\interior Q$ and $x\in V$, replacing $x$ by $tx$,
$t\in\R$, shows that $v^Tx=0$. Hence $v\in V^\perp$.
Taking $x=0$ and using $Q=\overline{\interior Q}$ from
Lemma~\ref{lem:cone-facts} gives $v\in Q^*$.
Thus $0\ne v\in Q^*\cap V^\perp$.
\end{proof}

\begin{lemma}[{\cite[Theorems~2.7 and~2.8]{cain1997}}]
\label{lem:alternative}
For $B\in\M{m}{n}$,
\begin{equation}
 B\notin S(H,D)
 \quad\Longleftrightarrow\quad
 \text{there exists }0\neq d\in D^*
 \text{ such that }B^Td\in-H^*.
 \label{eq:strict-alternative}
\end{equation}
\end{lemma}

\begin{proof}
Suppose that $B\notin S(H,D)$. The set
$U=\interior D-B(\interior H)$ is nonempty, open, and convex,
and does not contain zero.
Lemma~\ref{lem:open-cone} gives $0\ne d\in\R^m$ such that
\[
 d^T(c-Bh)\geq0
 \qquad(c\in\interior D,\ h\in\interior H).
\]
Scaling $c$ and $h$ separately towards zero, and then using the density
of the interiors, gives $d\in D^*$ and $B^Td\in-H^*$.

Conversely, suppose such a $d$ exists. If $h\in\interior H$ and
$Bh\in\interior D$, then \eqref{eq:strict-dual-positivity} gives
$0<d^TBh=(B^Td)^Th\leq0$, a contradiction.
\end{proof}

\section{Row sets and row operators}
\label{sec:row-operators}

Let $\R_+=[0,\infty)$. For a proper cone $P\subseteq\R^n$, write
\[
 \mathcal R_P=S(P,\R_+)
 =\{r\in\M{1}{n}:rx>0\text{ for some }x\in\interior P\}.
\]
Density of $\interior P$ in $P$ gives the useful equivalence
\begin{equation}
 r\notin\mathcal R_P
 \quad\Longleftrightarrow\quad
 rx\leq0\quad(x\in P)
 \quad\Longleftrightarrow\quad r^T\in-P^*.
 \label{eq:row-dual}
\end{equation}
Together with Lemma~\ref{lem:alternative}, this gives, for $B\in\M{m}{n}$,
\begin{equation}
 B\in S(H,D)
 \quad\Longleftrightarrow\quad
 d^TB\in\mathcal R_H
 \quad\text{for every }0\ne d\in D^*.
 \label{eq:semipositivity-row-test}
\end{equation}

\begin{lemma}
\label{lem:row-criterion}
For $0\neq u\in C^*$,
\begin{equation}
 \{u^TA:A\in S(K,C)\}=\mathcal R_K.
 \label{eq:source-rows}
\end{equation}
For $Y\in\M{n}{n}$,
\begin{equation}
 \{rY:r\in\mathcal R_K\}\subseteq\mathcal R_H
 \quad\Longleftrightarrow\quad
 Y\in\GL_n(\R)\ \text{ and }\ Y^{-1}K\subseteq H.
 \label{eq:row-criterion}
\end{equation}
\end{lemma}

\begin{proof}
If $A\in S(K,C)$, take $x\in\interior K$ with
$Ax\in\interior C$. Then $u^TAx>0$ by
\eqref{eq:strict-dual-positivity}, so $u^TA\in\mathcal R_K$.
For the reverse inclusion, choose $a_0\in\interior C$ and set
\[
 a=\frac{a_0}{u^Ta_0}.
\]
Then $a\in\interior C$ and $u^Ta=1$. Given $r\in\mathcal R_K$,
put $A=ar$. If $x\in\interior K$ satisfies $rx>0$, then
$Ax=(rx)a\in\interior C$, and $u^TA=r$.
This proves \eqref{eq:source-rows}.

By contraposition and \eqref{eq:row-dual}, the row inclusion in
\eqref{eq:row-criterion} is equivalent to
\[
 Y^Tv\in H^* \quad\Longrightarrow\quad v\in K^*
 \qquad(v\in\R^n).
\]
If $v\in\ker Y^T$, this implication gives $\pm v\in K^*$.
Since $K^*$ is pointed, $v=0$, and hence $Y$ is invertible.
For invertible $Y$, the displayed implication is equivalent to
\[
 H^*\subseteq Y^TK^*=(Y^{-1}K)^*,
\]
which, by the bipolar theorem, is equivalent to $Y^{-1}K\subseteq H$.
This proves \eqref{eq:row-criterion}.
\end{proof}

\begin{lemma}
\label{lem:row-rigidity}
Let $W\subseteq\M{m}{n}$ be a linear subspace such that
$W\cap S(K,C)=\varnothing$. Then $\operatorname{codim}W\ge n$,
with equality if and only if
\[
 W=\{A\in\M{m}{n}:u^TA=0\}
\]
for some $0\neq u\in C^*$.
In particular, if a linear map $T:\M{m}{n}\to\M{1}{n}$ satisfies
$T(A)\neq0$ for every $A\in S(K,C)$, then there are
$0\neq u\in C^*$ and $Y\in\GL_n(\R)$ such that
\begin{equation}
 T(A)=u^TAY\quad(A\in\M{m}{n}),\qquad
 \rank T=n,\qquad
 \ker T=\{A:u^TA=0\}.
 \label{eq:row-rigidity}
\end{equation}
\end{lemma}

\begin{proof}
Use the Frobenius inner product
$\langle B,A\rangle=\tr(B^TA)$ on $\M{m}{n}$, and set $Z=W^\perp$.
For $x\in\interior K$, the subspace
$Wx=\{Ax:A\in W\}$ avoids $\interior C$; otherwise an element of
$W$ would be semipositive. By Lemma~\ref{lem:annihilator}, there is
$0\neq u_x\in C^*$ annihilating $Wx$. Therefore
\[
 \langle u_xx^T,A\rangle=u_x^TAx=0\quad(A\in W),
 \qquad u_xx^T\in Z.
\]

Choose a basis $x_1,\ldots,x_n$ in $\interior K$, and write
$u_i=u_{x_i}$. Let $t_j\in\R^n$ satisfy
$x_i^Tt_j=\delta_{ij}$. If
$\sum_i\alpha_i u_ix_i^T=0$, right multiplication by $t_j$ gives
$\alpha_ju_j=0$, hence $\alpha_j=0$. Thus the matrices $u_ix_i^T$
are linearly independent, so $\operatorname{codim}W=\dim Z\ge n$.

Suppose equality holds. Then these matrices form a basis of $Z$.
Put $x_* =\sum_i x_i\in\interior K$, and choose
$0\neq u_*\in C^*$ with $u_*x_*^T\in Z$. Write
\[
 u_*x_*^T=\sum_i c_i u_ix_i^T.
\]
Right multiplication by $t_j$ yields $u_*=c_ju_j$ for every $j$.
Each $c_j$ is nonzero. With $u=u_*$, we obtain
\begin{equation}
 Z=\{uv^T:v\in\R^n\},\qquad
 W=\{A:u^TA=0\}.
 \label{eq:row-kernel}
\end{equation}
Conversely, every subspace of this form has codimension $n$ and
avoids $S(K,C)$ by \eqref{eq:strict-dual-positivity}.

For the assertion about $T$, take $W=\ker T$. Then
$W\cap S(K,C)=\varnothing$ and
$\operatorname{codim}W=\rank T\le n$. The preceding result gives
$\rank T=n$ and \eqref{eq:row-kernel}.
For each $j$, the $j$th coordinate functional of $T$ vanishes
on $W$, so \eqref{eq:row-kernel} gives $y_j\in\R^n$ such that
\[
 (T(A))_j=\langle uy_j^T,A\rangle=u^TAy_j.
\]
Hence $T(A)=u^TAY$, where $Y=[y_1\ \cdots\ y_n]$.
Since $\rank T=n$, $Y$ is invertible.
This proves \eqref{eq:row-rigidity}.
\end{proof}

\begin{lemma}
\label{lem:row-preserver}
A linear map $T:\M{m}{n}\to\M{1}{n}$ satisfies
$T(S(K,C))\subseteq\mathcal R_H$ if and only if there are
$0\neq u\in C^*$ and $Y\in\GL_n(\R)$ such that
\[
 T(A)=u^TAY\quad(A\in\M{m}{n}),\qquad Y^{-1}K\subseteq H.
\]
\end{lemma}

\begin{proof}
This follows from Lemmas~\ref{lem:row-criterion} and~\ref{lem:row-rigidity}.
\end{proof}

\section{Algebraic structure of preservers}
\label{sec:coefficient-assembly}

\begin{proposition}
\label{prop:algebraic-dichotomy}
Suppose that the linear map $L:\M{m}{n}\to\M{m}{n}$ satisfies
\[
 L(S(K,C))\subseteq S(H,D).
\]
At least one of the following representations holds:
\begin{enumerate}
 \item $L(A)=XAY$, where $0\neq X\in\M{m}{m}$ and
 $Y\in\GL_n(\R)$;
 \item $L(A)=\Psi(u^TA)$, where $0\neq u\in C^*$ and
 $\Psi:\M{1}{n}\to\M{m}{n}$ is linear.
\end{enumerate}
\end{proposition}

\begin{proof}
For $0\neq d\in D^*$, define $T_d(A)=d^TL(A)$.
By \eqref{eq:semipositivity-row-test},
$T_d(S(K,C))\subseteq\mathcal R_H$.

Choose a basis $d_1,\ldots,d_m$ of $\R^m$
inside $\interior D^*$, and put $T_i=T_{d_i}$.
Both $T_i$ and $T_i+T_j$ satisfy the hypotheses of
Lemma~\ref{lem:row-preserver}, since $d_i+d_j\in\interior D^*$.
Thus
\begin{equation}
 T_i(A)=u_i^TAY_i,
 \qquad 0\neq u_i\in C^*,\quad Y_i\in\GL_n(\R),
 \label{eq:individual-row-factors}
\end{equation}
and every pairwise sum has a representation of the same kind.

Suppose that the $u_i$ are not all collinear, and choose an
independent pair $u_i,u_j$. Write
\[
 u_i^TAY_i+u_j^TAY_j=u_{ij}^TAY_{ij}.
\]
Substituting $A=xy^T$ and letting $y\in\R^n$ vary gives
\[
 (u_i^Tx)Y_i+(u_j^Tx)Y_j=(u_{ij}^Tx)Y_{ij}
 \qquad(x\in\R^m).
\]
Choosing $x$ so that $(u_i^Tx,u_j^Tx)$ equals $(1,0)$ or $(0,1)$
shows that $Y_i$ and $Y_j$ are nonzero scalar multiples of $Y_{ij}$.
For every $k$, the nonzero vector $u_k$ is independent of at least
one of $u_i,u_j$, so the same argument places $Y_k$ on this line.
Thus either all $u_i$ are collinear or all $Y_i$ are collinear.

Let $R\in\M{m}{m}$ have rows $d_i^T$; then $R$ is invertible and the
$i$th row of $RL(A)$ is $T_i(A)$.
If the $Y_i$ are collinear, write $Y_i=c_iY$, where $Y=Y_1$ is
invertible and $c_i\neq0$. Let $U\in\M{m}{m}$ have rows
$c_iu_i^T$. Then
\[
 RL(A)=UAY,\qquad L(A)=XAY,\qquad X=R^{-1}U.
\]
The nonzero rows of $U$ imply $X\neq0$.

If the $u_i$ are collinear, take $u=u_1$ and write
$u_i=\lambda_i u$. Define
\[
 \Psi(r)=R^{-1}
 \begin{pmatrix}
  \lambda_1 rY_1\\
  \vdots\\
  \lambda_m rY_m
 \end{pmatrix}.
\]
This is linear, and \eqref{eq:individual-row-factors} yields
$L(A)=\Psi(u^TA)$.
\end{proof}

\section{Cone criteria for the two forms}
\label{sec:cone-criteria}

\begin{proposition}
\label{prop:product-criterion}
Let $X_0\in\M{m}{m}$ and $Y_0\in\M{n}{n}$.
The map $A\mapsto X_0AY_0$ sends $S(K,C)$ into
$S(H,D)$ if and only if $Y_0$ is invertible and there is a
unique $\varepsilon\in\{1,-1\}$ such that
\begin{equation}
 F=\varepsilon X_0,\qquad G=\varepsilon Y_0,
 \qquad F(\interior C)\subseteq\interior D,
 \qquad G^{-1}K\subseteq H.
 \label{eq:oriented-product-criterion}
\end{equation}
\end{proposition}

\begin{proof}
Suppose that $L(A)=X_0AY_0$ preserves semipositivity.
For every $0\ne d\in D^*$, the row map $T_d(A)=d^TL(A)$ sends
$S(K,C)$ into $\mathcal R_H$. Lemma~\ref{lem:row-preserver} gives
$0\ne u_d\in C^*$ and $Z_d\in\GL_n(\R)$, with
$Z_d^{-1}K\subseteq H$, such that
\[
 d^TX_0AY_0=u_d^TAZ_d\qquad(A\in\M{m}{n}).
\]
Taking $A=ar$ with $u_d^Ta=1$ shows that $Z_d=c_dY_0$ for some
$c_d\ne0$. Thus $Y_0$ is invertible, and the same identity gives
$X_0^Td=c_du_d$.

All $c_d$ have the same sign: otherwise the nonzero cone
$Y_0^{-1}K$ and its negative would both lie in the pointed cone $H$.
Let this sign be $\varepsilon$, and put $F=\varepsilon X_0$ and
$G=\varepsilon Y_0$. Then $G^{-1}K\subseteq H$ and
\[
 F^Td=|c_d|u_d\in C^*\setminus\{0\}
 \qquad(0\ne d\in D^*).
\]
Hence $d^TFa>0$ for every $a\in\interior C$ and $0\ne d\in D^*$,
so $F(\interior C)\subseteq\interior D$ by the dual characterization
of cone interiors. The same pointedness argument shows that
$\varepsilon$ is unique.

Conversely, assume \eqref{eq:oriented-product-criterion} and take
$A\in S(K,C)$. Choose $x\in\interior K$ with
$Ax\in\interior C$. The set $G^{-1}(\interior K)$ is open and
contained in $H$, so $h=G^{-1}x\in\interior H$. We then have
\[
 X_0AY_0h=FAGh=FAx\in\interior D.
\]
This proves preservation.
\end{proof}

\begin{proposition}
\label{prop:degenerate-criterion}
Fix $0\neq u\in C^*$ and a linear map
$\Psi:\M{1}{n}\to\M{m}{n}$, and let
$\mathcal Y(d)=\mathcal Y_\Psi(d)$ be the pencil defined in
\eqref{eq:pencil-definition}.
The map $L(A)=\Psi(u^TA)$ preserves semipositivity if and only if
\begin{equation}
 \mathcal Y(d)\in\GL_n(\R),\qquad \mathcal Y(d)^{-1}K\subseteq H
 \qquad\text{for every }0\neq d\in D^*.
 \label{eq:degenerate-criterion}
\end{equation}
Under these conditions, $\Psi$ is injective and $\rank L=n$.
\end{proposition}

\begin{proof}
Using \eqref{eq:semipositivity-row-test}, Lemma~\ref{lem:row-criterion},
and the definition of the pencil, we obtain
\[
\begin{aligned}
 L(S(K,C))\subseteq S(H,D)
 &\iff \Psi(\mathcal R_K)\subseteq S(H,D)\\
 &\iff \{r\mathcal Y(d):r\in\mathcal R_K\}\subseteq\mathcal R_H
       \quad(0\ne d\in D^*)\\
 &\iff \mathcal Y(d)\in\GL_n(\R),\quad
       \mathcal Y(d)^{-1}K\subseteq H
       \quad(0\ne d\in D^*).
\end{aligned}
\]

Under these conditions, choose $0\neq d_0\in D^*$. If $\Psi(r)=0$, then $r\mathcal Y(d_0)=0$;
invertibility of $\mathcal Y(d_0)$ gives $r=0$. Thus $\Psi$ is injective.
Since $P_u:A\mapsto u^TA$ is onto $\M{1}{n}$,
\begin{equation}
 \operatorname{im}L=\operatorname{im}\Psi,
 \qquad \rank L=n.
 \label{eq:degenerate-rank}
\end{equation}
\end{proof}

\section{Classification of semipositivity preservers}
\label{sec:classification-proof}

\begin{proof}[Proof of Theorem~\ref{thm:classification}]
Suppose first that $L$ preserves semipositivity.
Proposition~\ref{prop:algebraic-dichotomy} gives a product representation
or a degenerate representation. In the product case,
Proposition~\ref{prop:product-criterion} supplies the unique simultaneous
sign adjustment, which places $L$ in class \textnormal{(I)}. In the
degenerate case, Proposition~\ref{prop:degenerate-criterion} gives
exactly the pencil condition of class \textnormal{(II)}. Conversely,
these two propositions show that every map in either class preserves
semipositivity.

For a map in class \textnormal{(I)}, right multiplication by $Y$ is a
linear isomorphism of the source space. The image of $A\mapsto XA$
consists of matrices whose $n$ columns independently range over
$\operatorname{im}X$. Hence
\begin{equation}
 \rank L=n\rank X.
 \label{eq:product-rank}
\end{equation}

We next identify the intersection of the two classes. If $L$ belongs
to both classes, part~\textup{(i)} gives $n\rank X=\rank L=n$,
so $\rank X=1$. Conversely, suppose that $L$ belongs to class
\textnormal{(I)}, with factors $X,Y$ as in \eqref{eq:product-class},
and that $\rank X=1$. Fix $0\neq d_0\in D^*$ and put $u=X^Td_0$.
For every $a\in\interior C$, \eqref{eq:strict-dual-positivity}
gives $u^Ta=d_0^TXa>0$, so $0\neq u\in C^*$.
Since $\rank X=1$, we have $X=wu^T$ for some $w\in\R^m$.
Thus $L(A)=\Psi(u^TA)$, where $\Psi(r)=wrY$.
Since $L$ preserves semipositivity,
Proposition~\ref{prop:degenerate-criterion} places $L$ in
class \textnormal{(II)}, proving part~\textup{(ii)}.
It follows that class \textnormal{(II)} consists precisely of the
preserving maps of rank $n$.
\end{proof}

Since the pencil $\mathcal Y_\Psi(d)$ is linear in $d$, one might try
to verify \eqref{eq:full-dual-condition} only on the extreme rays of $D^*$
when $D^*$ is polyhedral. The following example shows that this is
insufficient.

\begin{example}[Testing extreme rays]\label{ex:ray-test}
Take $K=H=C=D=\R_+^2$. Let
\[
 P=\begin{pmatrix}0&1\\1&0\end{pmatrix},\qquad
 \Psi(r)=\begin{pmatrix}r\\rP\end{pmatrix}
 \quad(r\in\M{1}{2}),
\]
and set $L(A)=\Psi(e_1^TA)$, where $e_i$ denotes the $i$th standard
basis vector. For $d=(d_1,d_2)^T$,
\[
 d^T\Psi(r)=d_1r+d_2rP=r(d_1I_2+d_2P),
\]
so the associated pencil is $\mathcal Y(d)=d_1I_2+d_2P$.

On the two extreme rays of $D^*=\R_+^2$, we have
\[
 \mathcal Y(te_1)^{-1}=t^{-1}I_2\ge0,\qquad
 \mathcal Y(te_2)^{-1}=t^{-1}P\ge0\qquad(t>0).
\]
At the interior point $d=e_1+e_2$, however,
\[
 \mathcal Y(e_1+e_2)=I_2+P=\begin{pmatrix}1&1\\1&1\end{pmatrix}
\]
is singular, violating \eqref{eq:full-dual-condition}.
\end{example}

Theorem~\ref{thm:classification} therefore yields the following
characterization of invertible preservers.

\begin{corollary}\label{cor:invertible}
A real linear operator $L$ on $\M{m}{n}$ is invertible and satisfies
\eqref{eq:preservation} if and only if there are $X$ and $Y$ such that
\begin{equation}\label{eq:invertible-form}
 L(A)=XAY\quad\text{for every }A,\qquad
 X\in\GL_m(\R),\quad Y\in\GL_n(\R),
\end{equation}
and
\begin{equation}\label{eq:invertible-cones}
 XC\subseteq D,\qquad Y^{-1}K\subseteq H.
\end{equation}
For any prescribed product representation of such an operator, both
factors are invertible, and a unique simultaneous choice of sign makes
\eqref{eq:invertible-cones} hold.
\end{corollary}

\begin{proof}
If $L$ is an invertible preserver, then $\rank L=mn>n$.
Theorem~\ref{thm:classification} gives the product form,
and part~\textup{(i)} gives $mn=n\rank X$. Thus $X$ is invertible.
The inclusion $XC\subseteq D$ follows from
$X(\interior C)\subseteq\interior D$ by continuity and closedness of $D$.

Conversely, suppose that \eqref{eq:invertible-form} and
\eqref{eq:invertible-cones} hold. Since $X$ is invertible,
$X(\interior C)$ is open and contained in $D$, and hence lies in
$\interior D$. Theorem~\ref{thm:classification} therefore gives
preservation, and the inverse of $L$ is $B\mapsto X^{-1}BY^{-1}$.
The assertion about a prescribed representation follows from
Proposition~\ref{prop:product-criterion}, \eqref{eq:product-rank},
and the equivalence of the closed-cone and interior inclusions for
invertible $X$.
\end{proof}

\section{The orthant case and preservers from \texorpdfstring{$\M{2}{2}$ to $\M{1}{3}$}{M(2,2) to M(1,3)}}\label{sec:examples}

\begin{corollary}\label{cor:orthants}
Let $m,n\ge2$. For the standard nonnegative orthants,
a linear map $L:\M{m}{n}\to\M{m}{n}$ preserves semipositivity if
and only if it admits at least one of the following representations:
\begin{enumerate}[label=\textup{(\roman*)}]
\item $L(A)=XAY$, where $X\ge0$ has no zero row,
$Y\in\GL_n(\R)$, and $Y^{-1}\ge0$;
\item $L(A)=\Psi(u^TA)$, where $0\ne u\ge0$, $\Psi$ is linear,
and its pencil $\mathcal Y(d)=\mathcal Y_\Psi(d)$ satisfies
\[
 \mathcal Y(d)\in\GL_n(\R),\qquad \mathcal Y(d)^{-1}\ge0
 \qquad(0\ne d\in\R_+^m).
\]
\end{enumerate}
The invertible preservers are exactly the maps in \textup{(i)} with
$X\in\GL_m(\R)$. For a prescribed product representation of a preserving map, the conditions
in \textup{(i)} hold after a unique simultaneous choice of sign.
\end{corollary}

\begin{proof}
The nonnegative orthant is its own dual. A matrix $X$ maps every
strictly positive vector to a strictly positive vector if and only if
$X\ge0$ and $X$ has no zero row. Indeed, a negative entry in row $i$ and
column $j$ would make the $i$th output coordinate negative when the
$j$th input coordinate is sufficiently large and all other input
coordinates equal one. A zero row is also excluded. Conversely, every nonzero
nonnegative row has positive inner product with every strictly
positive column. For an invertible matrix $Y$, the condition
$Y^{-1}\R_+^n\subseteq\R_+^n$ is equivalent to nonnegativity of
each column $Y^{-1}e_j$, hence to $Y^{-1}\ge0$.
The classification follows from Theorem~\ref{thm:classification};
the invertibility and sign assertions follow from
Corollary~\ref{cor:invertible} and Proposition~\ref{prop:product-criterion}.\qedhere
\end{proof}

\medskip
\noindent\textbf{Preservers from $\M{2}{2}$ to $\M{1}{3}$.}
Preservers between matrix spaces of different dimensions can have
considerably more intricate forms, as the following classification illustrates.
We now consider linear maps
$T:\M{2}{2}\to\M{1}{3}$ preserving orthant semipositivity.
For this discussion, write
\[
 \mathcal S_{a,b}=S(\R_+^b,\R_+^a),\qquad
 J_0=\begin{pmatrix}0&1\\-1&0\end{pmatrix}.
\]

\begin{theorem}\label{thm:additional-column}
A real linear map $T:\M{2}{2}\to\M{1}{3}$ satisfies
$T(\mathcal S_{2,2})\subseteq\mathcal S_{1,3}$ if and only if it
belongs to exactly one of the following three classes.
\begin{enumerate}[label=\textup{(\Roman*)},ref=\Roman*,beginpenalty=10000]
\item There are $0\ne u\in\R_+^2$ and
$Y\in\M{2}{3}$ such that
\[
 T(A)=(u^TA)Y,\qquad
 YN=I_2\quad\text{for some }N\in\M{3}{2},\ N\ge0.
\]
Equivalently, the condition on $Y$ is
$\R_+^2\subseteq Y\R_+^3$. In this case
$\ker T=\{A:u^TA=0\}$.

\item There are $0\ne u\in\R_+^2$, $0\ne w\in\R^2$, and
$G\in\GL_3(\R)$ whose first two columns are entrywise nonnegative,
such that
\[
 T(A)=F_{u,w}(A)G^{-1},\qquad
 F_{u,w}(A)=\bigl(u^TA,\ (J_0u)^TAJ_0w\bigr).
\]
In this case
$\ker T=\R\bigl((J_0u)w^T\bigr)$.

\item There are
\[
 P=\begin{pmatrix}\alpha&\beta\\\gamma&\delta\end{pmatrix}\ge0,
 \quad\det P\ne0,\quad G\in\GL_3(\R),
\]
such that $T(A)=Q_P(A)G^{-1}$. For
$A=\begin{pmatrix}a&b\\c&d\end{pmatrix}$, the map $Q_P$ is defined by
\[
 Q_P(A)=\bigl(\alpha a+\gamma c,
 \beta a+\alpha b+\delta c+\gamma d,
 \beta b+\delta d\bigr).
\]
The rows $(\ell_i,m_i,n_i)$ of $G$ must satisfy
\begin{equation}\label{eq:binary-quadratic-condition}
 \ell_i\ge0,\qquad n_i\ge0,\qquad
 m_i\ge-2\sqrt{\ell_i n_i}\qquad(i=1,2,3).
\end{equation}
In this case $\ker T=\R(J_0P)$.
\end{enumerate}
Class \textup{(I)} consists of the preservers of rank two,
while classes \textup{(II)} and \textup{(III)} consist of those of
rank three whose kernels are spanned by singular and invertible
matrices, respectively.
\end{theorem}

\begin{proof}
Set $W=\ker T$. If $T$ preserves semipositivity, then
$W\cap\mathcal S_{2,2}=\varnothing$. Lemma~\ref{lem:row-rigidity}
gives $\operatorname{codim}W\ge2$, so $\rank T$ is either $2$ or $3$.
If $\rank T=2$, the same lemma yields $W=\{A:u^TA=0\}$
for some $0\ne u\ge0$. Hence $T(A)=(u^TA)Y$ for some
$Y\in\M{2}{3}$. By \eqref{eq:source-rows}, preservation is equivalent to
\[
 rY\le0\ \Longrightarrow\ r\le0.
\]
Separation of the finitely generated closed cone $Y\R_+^3$ shows
that this is equivalent to $\R_+^2\subseteq Y\R_+^3$.
Taking nonnegative preimages of $e_1,e_2$ gives $N\ge0$ with
$YN=I_2$. Conversely, this identity implies the displayed implication.
It also forces $\rank Y=2$. This completes case~\textnormal{(I)}.

Now let $\rank T=3$ and write $W=\R B$ with $B\ne0$.
If $B=vz^T$ has rank one, neither $B$ nor $-B$ is semipositive,
so $v_1v_2\le0$. Consequently
$B=(J_0u)w^T$ for some $0\ne u\ge0$ and $w\ne0$.
Direct calculation gives
\[
 \ker F_{u,w}=\R\bigl((J_0u)w^T\bigr).
\]
Indeed, $u^TA=0$ implies $A=(J_0u)\xi^T$, and the third
coordinate of $F_{u,w}(A)$ is $\|u\|^2\xi^TJ_0w$.
Thus $F_{u,w}$ and $T$ are surjective with the same kernel,
and $F_{u,w}(A)=T(A)G$ for some $G\in\GL_3(\R)$.

We use the exact image identity
\begin{equation}\label{eq:singular-kernel-image}
 F_{u,w}(\mathcal S_{2,2})=\mathcal S_{1,2}\times\R.
\end{equation}
Here $\mathcal S_{1,2}$ consists of all $1\times 2$ rows with at least one positive entry.
One inclusion follows from \eqref{eq:source-rows}. For the other,
choose $A_0$ with $F_{u,w}(A_0)=(r,s)$, where $r\in\mathcal S_{1,2}$.
Choose $x>0$ with $rx>0$ and $w^Tx\ne0$; the excluded hyperplane
cannot contain the nonempty open set defined by $x>0$ and $rx>0$.
Let $\mathbf1_q$ denote the all-ones column in $\R^q$; we omit the
subscript when the dimension is clear. The positive vector
$z=(rx)\mathbf1_2/(u_1+u_2)$ satisfies
$u^Tz=u^TA_0x$. Therefore $z-A_0x\in\R J_0u$, so for a suitable
$t\in\R$ the matrix $A=A_0+t(J_0u)w^T$ satisfies $Ax=z$ and
$F_{u,w}(A)=F_{u,w}(A_0)$.
This proves \eqref{eq:singular-kernel-image}.
By this identity, the map $A\mapsto F_{u,w}(A)G^{-1}$ preserves semipositivity
if and only if, for every row $t\in\M{1}{3}$ with $t\le0$, the first
two coordinates of $tG$ are nonpositive. Testing $t=-e_i^T$ shows
that this is equivalent to nonnegativity of the first two columns
of $G$.
This completes case~\textnormal{(II)}.

If $B$ is invertible, $B(\interior\R_+^2)$ is open and connected, and
contains no vector whose coordinates are both positive or both negative.
It cannot meet either coordinate axis, since it is open.
It therefore lies in one of the two opposite-sign open quadrants.
After replacing $B$ by $-B$ if necessary, the first row of $B$ is
nonnegative and its second row is nonpositive. Thus $B=J_0P$
for an invertible $P\ge0$.
With $v(x)=(x_1^2,x_1x_2,x_2^2)^T$, one has
\begin{equation}\label{eq:quadratic-witness}
 Q_P(A)v(x)=(Px)^TAx.
\end{equation}
The equation $Q_P(A)=0$ says that $P^TA$ is skew-symmetric,
so $\ker Q_P=\R(J_0P)$. Hence $Q_P(A)=T(A)G$ for some
$G\in\GL_3(\R)$.

For $x>0$, put $y=Ax$ and $z=Px>0$. Since
$(A+tJ_0P)x=(y_1+tz_2,y_2-tz_1)^T$, we have
\begin{equation}\label{eq:invertible-kernel-fiber}
 \exists t\in\R:\ (A+tJ_0P)x>0
 \quad\Longleftrightarrow\quad (Px)^TAx>0.
\end{equation}
Choose $A_i$ with $T(A_i)=-e_i^T$. No matrix in
$A_i+\R J_0P$ is semipositive. Equations
\eqref{eq:quadratic-witness}--\eqref{eq:invertible-kernel-fiber}
therefore give $e_i^TGv(x)\ge0$ for every $x>0$.
Conversely, suppose that $Gv(x)\ge0$ for every $x>0$. If
$A\in\mathcal S_{2,2}$ and $x>0$ satisfies $Ax>0$, then
\[
 T(A)Gv(x)=(Px)^TAx>0,
\]
so $T(A)$ has a positive entry. Preservation is consequently
equivalent to nonnegativity of each quadratic form
$\ell_i s^2+m_i st+n_i t^2$ for $s,t>0$.
This is precisely \eqref{eq:binary-quadratic-condition}:
limits at the axes give $\ell_i,n_i\ge0$, and minimizing over
$s/t>0$ gives the remaining inequality, with $m_i\ge0$ if
$\ell_i n_i=0$. Sufficiency follows from
\[
 \ell s^2+mst+nt^2
 = (\sqrt\ell\,s-\sqrt n\,t)^2
   +(m+2\sqrt{\ell n})st.
\]
This completes case~\textnormal{(III)}. Finally, the rank
and the matrix rank of a nonzero kernel generator distinguish
the three cases, which exhaust all possibilities.
\end{proof}

\begin{remark}
\label{rem:additional-column}
For $T:\M{2}{2}\to\M{1}{3}$, every rectangular product
$T(A)=XAY$, with $X\in\M{1}{2}$ and $Y\in\M{2}{3}$, has
rank at most two. Thus both rank-three classes in
Theorem~\ref{thm:additional-column} lie outside the product form.
They also differ in whether the third column can be discarded
after a fixed linear recombination of output columns.
More precisely, call $T$ \emph{compressible to two columns} if
there is $R\in\M{3}{2}$ such that $A\mapsto T(A)R$ preserves
semipositivity. Class~\textnormal{(I)} is compressible with $R=N$,
and class~\textnormal{(II)} with $R=G\binom{I_2}{0_{1\times2}}$;
both choices are nonnegative and give $T(A)R=u^TA$.
Class~\textnormal{(III)} admits no such real $R$. Otherwise
Lemma~\ref{lem:row-rigidity} would give
$\ker(A\mapsto T(A)R)=\{A:u^TA=0\}$ for some $0\ne u\ge0$,
forcing $u^TJ_0P=0$, contrary to invertibility of $J_0P$.
\end{remark}

\section{Proof of the minimal semipositivity classification}
\label{sec:minimal-semipositivity}

The key step is to prove that $\det L(A)=c\det A$ for some $c\ne0$.
We use the following polynomial lemma.

\begin{lemma}\label{lem:msp-polynomial}
Let $n\ge2$, and let $F$ be a real polynomial of total degree at most $n$
in the entries of $A\in\M{n}{n}$. Suppose that $F(A)=0$ whenever
$\rank A=n-1$ and
\[
 Au=0,\qquad v^TA=0\qquad\text{for some }u,v>0.
\]
Then there is $c\in\R$ such that $F(A)=c\det A$ for every
$A\in\M{n}{n}$.
\end{lemma}

\begin{proof}
Consider the parameter set
\[
 U=\left\{(B,p,q):
 \begin{array}{l}
 B\in\GL_{n-1}(\R),\quad p,q\in\R^{n-1},\\
 -B^{-1}p>0,\quad -B^{-T}q>0
 \end{array}\right\}.
\]
It is open and contains $(I_{n-1},-\mathbf1,-\mathbf1)$.
For $(B,p,q)\in U$, the matrix
\[
 A(B,p,q)=
 \begin{pmatrix}
 B&p\\ q^T&q^TB^{-1}p
 \end{pmatrix}
\]
has rank $n-1$ and positive right and left null vectors
\[
 \binom{-B^{-1}p}{1},
 \qquad
 \binom{-B^{-T}q}{1}.
\]
Hence $F(A(B,p,q))=0$.

Regard $d$ and the entries of $B,p,q$ as independent variables, and set
\[
 s=\det B,\qquad r=q^T\operatorname{adj}(B)p,\qquad
 \widehat F(d)=F\begin{pmatrix}B&p\\q^T&d\end{pmatrix}
             =\sum_{j=0}^{n}a_jd^j,
\]
where each $a_j$ is a polynomial in the entries of $B,p,q$.
The polynomial
\[
 R=\sum_{j=0}^{n}a_jr^js^{n-j}
\]
vanishes on $U$, since $R=s^n\widehat F(r/s)=s^nF(A(B,p,q))$ there.
Thus $R$ is identically zero. Put
\[
 f=sd-r=\det\begin{pmatrix}B&p\\q^T&d\end{pmatrix}.
\]
The identity
\[
 s^n\widehat F(d)-R
 =\sum_{j=1}^{n}a_js^{n-j}\bigl((sd)^j-r^j\bigr)
\]
shows that $f$ divides $s^n\widehat F(d)$, because every term on the
right is divisible by $sd-r$.
The generic determinant $f$ is irreducible and hence prime in the real
polynomial ring in $d$ and the entries of $B,p,q$.
Since $\deg f=n>\deg s=n-1$, it does not divide $s$ and therefore
must divide $\widehat F(d)$.
The total degree bound on $F$ now gives $F(A)=c\det A$ for some $c\in\R$.
\end{proof}

\begin{proof}[Proof of Theorem~\ref{thm:msp-classification}]
Suppose that $L(\mathcal I_n)\subseteq\mathcal I_n$, and set
$F(A)=\det L(A)$. This is a nonzero homogeneous polynomial of degree $n$,
since $I\in\mathcal I_n$ implies $F(I)\ne0$.

Let $A$ have rank $n-1$ and suppose that
\[
 Au=0,\qquad v^TA=0,\qquad u,v>0.
\]
Then $\operatorname{adj}(A)=\alpha uv^T$ for some $\alpha\ne0$.
Put $E=e_ne_n^T$. Expansion in the $(n,n)$-entry gives
$\det(A+tE)=t\alpha u_nv_n$. Hence
\begin{equation}\label{eq:msp-perturbation}
 t(A+tE)^{-1}
 =\frac{\operatorname{adj}(A+tE)}{\alpha u_nv_n}
 \longrightarrow\frac{uv^T}{u_nv_n}>0
 \qquad(t\downarrow0).
\end{equation}
Applying the same argument to $-A$, we obtain
$A+tE,-A+tE\in\mathcal I_n$ for all sufficiently small $t>0$.
If $L(A)$ were invertible, continuity of inversion as $t\downarrow0$
would give
\[
 L(A)^{-1}\ge0,\qquad -L(A)^{-1}\ge0,
\]
because $L(A+tE),L(-A+tE)\in\mathcal I_n$.
This is impossible. Hence $F(A)=0$.
Applying Lemma~\ref{lem:msp-polynomial}, we obtain
\begin{equation}\label{eq:msp-determinant}
 \det L(A)=c\det A\quad(A\in\M{n}{n}),
 \qquad c=\det L(I)\ne0.
\end{equation}

It follows that $L(A)$ is invertible if and only if $A$ is invertible.
The classical structure theorem for linear maps preserving
nonsingularity in both directions
\cite[Theorem~1(ii)]{deseguinspazzis2010}
therefore gives
\[
 L(A)=XAY\quad\text{for every }A,
 \qquad\text{or}\qquad
 L(A)=XA^TY\quad\text{for every }A,
\]
with $X,Y\in\GL_n(\R)$. The factor criteria in
\cite[Theorems~2.10 and~2.11]{choudhurykannansivakumar2018}
imply that $X^{-1},Y^{-1}$ are either both nonnegative or both
nonpositive. Replacing $X,Y$ simultaneously by $-X,-Y$ in the latter
case proves necessity.

Conversely, if $A^{-1}\ge0$ and $X^{-1},Y^{-1}\ge0$, then
\[
 (XAY)^{-1}=Y^{-1}A^{-1}X^{-1}\ge0,\qquad
 (XA^TY)^{-1}=Y^{-1}(A^{-1})^TX^{-1}\ge0.
\]
Both types of maps send $\mathcal I_n$ into itself and are invertible.
\end{proof}

\bigskip
{\small
\noindent\textbf{Yulong Tian}\par
\noindent\textsc{Department of Mathematics, Northeast Forestry University,
Harbin, P.~R.~China}\par
\noindent\textit{Email address:}
\href{mailto:tylmath@nefu.edu.cn}{\texttt{tylmath@nefu.edu.cn}}\par
\medskip
\noindent\textbf{Jinli Xu}\par
\noindent\textsc{Department of Mathematics, Northeast Forestry University,
Harbin, P.~R.~China}\par
\noindent\textit{Email address:}
\href{mailto:jclixv@qq.com}{\texttt{jclixv@qq.com}}\par
}

\end{document}